\documentclass[a4paper, 11pt] {article}
\usepackage{amsfonts, amsmath, amssymb, amsthm,url,dsfont,mathtools}
\usepackage{graphicx}
\usepackage[english] {babel}
\usepackage{enumitem}
\usepackage[draft]{fixme}
\usepackage{array}
\usepackage{cellspace}
\usepackage[ps, all,cmtip, dvips]{xy}
\usepackage{tikz}

\usepackage{geometry}
\newtheorem{thm}{Theorem}[section]
\newtheorem{cor}[thm]{Corollary}
\newtheorem{lem}[thm]{Lemma}
\newtheorem{defi}[thm]{Definition}

\theoremstyle{definition}
\newtheorem{rem}[thm]{Remark}
\newtheorem{ex}[thm]{Example}

\DeclareMathOperator{\II}{\textit{II}}

\newcommand{\R}{\mathbb{R}}
\newcommand{\supp}{\textrm{supp}}

\newcommand{\La}{\mathbb{L}}

\newcommand{\eps}{\varepsilon}

\DeclareMathOperator{\indre}{int}

\DeclareMathOperator{\tr}{Tr}
\newcommand{\Ha}{\mathcal{H}}

\begin{document}

\section*{Estimation of Minkowski tensors from digital grey-scale images}

\begin{abstract}
It was shown in \cite{am4} that local algorithms based on grey-scale images sometimes lead to asymptotically unbiased estimators for surface area and integrated mean curvature. This paper extends the results to estimators for Minkowski tensors. In particular, asymptotically unbiased local algorithms for estimation of all volume and surface tensors and certain mean curvature tensors are given. This requires an extension of the asymptotic formulas of \cite{am4} to estimators with position dependent weights. 
\end{abstract}

\section{Introduction}
Minkowski tensors \cite{hug,schuster} are generalizations of Minkowski functionals \cite{schneider}, associating to a compact convex body $X\subseteq \R^d$ a symmetric tensor, rather than a scalar. They  carry information about shape features of $X$ such as position, anisotropy, and eccentricity. They are used as shape descriptors in statistical physics.  For instance, in \cite{aste} Minkowski tensors are used to detect anisotropy in spherical bead packs. See e.g.\ \cite{mickel} for an overview.

Since the data is often of digital nature, there is a need for fast digital algorithms to estimate tensors. Such algorithms are suggested in \cite{turk, mecke}. These algorithms are all of local type, see \cite{am3}, based on black-and-white images. 

It is well known that local algorithms for Minkowski functionals based on black-and-white images are generally biased \cite{kampf2,am3}. The situation seems to be the same for most Minkowski tensors. Since most black-and-white images occur as thresholded grey-scale images, the focus has switched to algorithms based directly on grey-scale images without thresholding where recent results \cite{am4} show the existence of asymptotically unbiased algorithms for surface area and integrated mean curvature. Grey-scale images and local estimators are explained in  Section \ref{gssec}.

Surface area and integrated mean curvature can be estimated using only $1\times \dotsm \times 1$ configurations, whereas larger $n \times \dots \times n$ configurations are needed in order to gain information about surface normals. Moreover, position dependent weights are needed in order to get information about position. 
This requires a slight extension of the known results about the asymptotic behaviour of local algorithms. These follow fairly easily from the technical lemmas in \cite{am4}. The theoretical results are given in Section \ref{1sec}. 

The estimation of Minkowski tensors is the topic of Section \ref{mink}. The formal definition of the tensors is given in Subsection \ref{tens}. The subsequent subsections introduce local estimators for volume, surface, and certain mean curvature tensors. The algorithms are asymptotically unbiased, i.e.\ they converge when the resolution tends to infinity and the point spread function (PSF) becomes concenteated near the boundary. In particular, we obtain a complete set of estimators for the Minkowski tensors in 2D. The algorithms require that the PSF is known; at least the knowledge of what a blurred halfspace looks like is required.  Moreover, the resolution has to be sufficiently high compared to the support of the PSF.

\section{Local estimators from grey-scale images} \label{gssec}

\subsection{Grey-scale images}
Let $X \subseteq \R^d$ be the compact set we are observing. 
We assume that the light coming from each point is spread out following a point spread function which is independent of the position of the point. Hence the light that reaches the observer is given by the intensity function
\begin{equation*}
\theta^{X} : \R^d \to [0,1]
\end{equation*}
where the intensity measured at $x\in \R^d$ is given by
\begin{equation*}
\theta^{X}(x)= \int_{X} \rho(z-x) dz.
\end{equation*}
In other words $\theta^{X}$ the convolution $\mathds{1}_X * \rho $ of the indicator function for $X$ with a PSF $\rho$.
The PSF is assumed to be a measurable function satisfying 
\begin{itemize}
\item $\rho \geq 0$.
\item $\int_{\R^d} \rho(z) dz =1$.
\end{itemize}
We say that a PSF is rotation invariant if $\rho(x)=\rho(|x|)$ depends only on $|x|$.

A digital grey-scale image is the restriction of $\theta^{X}$ to an observation lattice $\La$.
A change of resolution corresponds to a change of lattice from $\La$ to $a\La$ for some $a>0$. We assume that the precision of the measurements changes with resolution in such a way that the PSF corresponding to $a\La$ is
\begin{equation*}
\rho_a(x)=a^{-d}\rho(a^{-1}x),
\end{equation*}
see the discussion in \cite[Section 2.1]{am4}.
The corresponding intensity function is denoted
\begin{equation*}
\theta_a^{X}(x)= \int_{X} \rho_a(z-x) dz =a^{-d} \int_{X} \rho(a^{-1}(z-x)) dz.
\end{equation*}

In applications, the PSF is typically the Gaussian function \cite{kothe} or the Airy disk \cite{airy}. These are smooth and rotation invariant but do not have compact support. Another important example is $\rho_B = \Ha^d(B)^{-1}\mathds{1}_{B}$ where $B \subseteq \R^d$ is a compact set of non-zero finite volume $\Ha^{d}(B)$. In this case, we measure at each $z\in \La$ the fraction of $z+B$ covered by $X$.  Such PSF's have compact support but are not continuous.

\subsection{A blurred halfspace}
For $u\in S^{d-1}$ and $\alpha \in \R$ write $H_{\alpha, u}^- = \{x\in \R^d \mid \langle x, u \rangle \leq \alpha \}$ for the halfspace.  The intensity function associated to a halfspace in standard resolution will play a special role in the following. Hence we introduce the separate notation 
\begin{equation*}
\theta_u(t):=\theta^{H_{0,u}^-}_1(tu).
\end{equation*}

A geometric interpretation of $\theta_u$ is illustrated in Figure \ref{Hfig}.

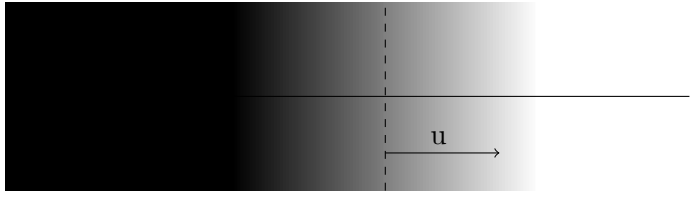
\begin{figure}\label{Hfig}
\begin{equation*}
\begin{tikzpicture}
\shade[left color=black,right color=black]  (0,0) rectangle (7,2.5);
\shade[left color=black,right color=white] (3,0) rectangle (7,2.5);
\draw[->] (5,.5) -- (6.5,.5);
\draw (0,1.25) -- (9,1.25);
\draw (5.7,.7) node{u};
\draw[dashed] (5,0) -- (5,2.5);
\end{tikzpicture}
\end{equation*}
\caption{A grey-scale image of a halfspace. The function $\theta_u$ measures how the grey-values change along the horizontal line going from right to left.}
\end{figure}

\begin{ex}
 If $\rho$ is the standard Gaussian
\begin{equation*}
\rho(x) = (2\pi)^{-\frac{d}{2}}e^{-\frac{1}{2}|x|^2},
\end{equation*} 
then 
\begin{equation*}
\theta_u(t) = \int_{u^\perp} \int_{-\infty}^{0} \rho(tu-z-su) ds dz = \Phi(-t)
\end{equation*}
where $\Phi$ is the distribution function for the standard 1-dimensional normal distribution.
\end{ex}

Note for later that
\begin{equation*}
\theta_a^{H_{0,u}^-}(ax)=\theta^{H_{0,u}^-}_1(\langle x, u \rangle u) =\theta_u(\langle x, u \rangle)
\end{equation*}
independently of $a$.

\subsection{Local algorithms in the grey-scale setting}\label{lag-s}

Let $\La$ be a lattice in $\R^d$ spanned by the ordered basis $v_1,\dots,v_d\in \R^d$ and let  $C_v=\bigoplus_{i=1}^d[0,v_i)$ be the fundamental cell of the lattice. As we shall later be scaling the lattice, we may as well assume that the volume $\det(v_1,\dots,v_d)$ of $C_v$ is 1. For $c\in \R^d$, we let $\La_c=\La+c$ denote the lattice translated by $c$ and $a\La_c$ the scaling of $\La_c$ by $a>0$. 

A fundamental $n\times \dotsm \times n$ lattice block is $C_{w,0}^n=(w+nC_v) \cap \La$ for some fixed $w\in \La$. More generally we consider its translations $C_{w,z}^n=z + C_{w,0}^n$ by $z\in \R^d$. We denote by $[0,1]^{C_{w,0}^n}$ the set of $n^d$-tuples of points in $[0,1]$ indexed by $C_{w,0}^n$. A point is written $\{\theta_s\}_{s\in C_{w,0}^n}$. The restriction of $\theta_a^X$ to $aC_{w,z}^n$ naturally defines a point in $[0,1]^{C_{w,0}^n}$ which we denote by $\Theta_a^X(az;aC^n_{w,0})=\{\theta_a^X(az+as))\}_{s\in C^n_{w,0}}$. 

\begin{defi}\label{greyest}
A local algorithm $\hat{\Phi}_q^f$ is an estimator of the form 
\begin{equation}\label{sumest}
\hat{\Phi}_q^{f}(X)=a^q \sum_{z \in \La_c} f(\Theta^X_a(az;aC_{w,0}^n),z)
\end{equation}
where $f:[0,1]^{C_{w,0}^n } \times \R^d \to \R$ is a Borel function. We assume that the support of $f$ is contained in $A\times \R^d$ where $A\subseteq (0,1)^{C_{w,0}^n}$ is compact and that $f$ is bounded on compact sets.
\end{defi}
The assumptions on $f $ make the sum \eqref{sumest} finite and $z\mapsto f ( \Theta^X_a(az;aC_{w,0}^n),z)$ integrable whenever $X$ is compact.

We assume that the lattice is stationary random, i.e.\ we consider the lattice $\La_c = \La +c$ where $c\in C_v$ is uniform random. Then the mean estimator is
\begin{equation}\label{mean}
E\hat{\Phi}^f_q(X)=a^q E\sum_{z\in \La_c} f(\Theta^X_a(az;aC_{w,0}^n),az) = a^{q-d}\int_{\R^d} f ( \Theta^X_a(z;aC_{w,0}^n),z) dz.
\end{equation}

As a natural convergence criterion, we take the following:
\begin{defi}
A local algorithm is an asymptotically unbiased estimator for $\Phi(X)$ if $\lim_{a\to 0} E\hat{\Phi}^f_q(X)=\Phi(X)$.
\end{defi}

\subsection{The relevant set-classes}
In order to prove the formulas, we need to make some assumptions on $X$. 
First some notation. 
For a closed set $X\subseteq \R^d$, we let $\text{exo}(X)$ denote the points in $\R^d$ not having a unique nearest point in $X$. Let $\xi_X : \R^d\backslash \text{exo}(X) \to X$ be the natural projection taking a point in $\R^d\backslash \text{exo}(X)$ to its nearest point in $X$. We define the normal bundle of $X$ to be the set
\begin{equation*}
N(X)=\big\{\big(x,\tfrac{z-x}{|z-x|}\big)\in X\times S^{d-1}\, \big|\, z\in \R^d\backslash (X\cup \text{exo}(X)),\, \xi_X(z)=x \big\}.
\end{equation*}
For $(x,u)\in N(X)$ we define the reach
\begin{equation*}
\delta(X;x,u)=\inf\{t\geq 0 \mid x+tu \in \text{exo}(X)\}\in (0,\infty].
\end{equation*}

Let $\Ha^k$ denote the $k$-dimensional Hausdorff measure. Following \cite{rataj}, we introduce the class of gentle sets:
\begin{defi}\label{gentle}
A closed set $X\subseteq \R^d $ is called gentle if
\begin{itemize}
\item[(i)] $\Ha^{d-1}(N(\partial X) \cap (B \times S^{d-1}))<\infty$ for any bounded Borel set $B\subseteq \R^d$.
\item[(ii)] For $\Ha^{d-1}$-almost all $x\in \partial X$ there exist two balls $B_{in},B_{out}\subseteq \R^d$ both containing $x$ and such that $B_{in}\subseteq X$, $\indre(B_{out})\subseteq \R^d \backslash X$.
\end{itemize}
\end{defi}
The condition (ii) in the definition means that for almost all $x\in \partial X$ there is a unique pair $(x,u(x))\in N(X)$ with $(x,u(x)),(x,-u(x))\in N(\partial X)$. This class is quite general, including for instance all $C^1$ manifolds and all polyconvex sets satisfying a certain regularity condition, see \cite{rataj}.

We shall also consider the subclass of $r$-regular sets:
\begin{defi}
A gentle set $X\subseteq \R^d$ is called $r$-regular for some $r>0$, if the balls  $B_{in}$ and $B_{out}$ exist for every $x\in \partial X$ and can be chosen to have radius $r$.
\end{defi}
Being $r$-regular is slightly weaker than being a $C^2$ manifold.

It can be proved \cite{federer},  that if $X$ is $r$-regular, then $\partial X$ is a $C^1$ manifold with $\Ha^{d-1}$-a.e.\ dif\-fe\-ren\-tia\-ble normal vector field $u$. Thus its principal curvatures $k_1,\dots, k_{d-1}$, corresponding to the orthogonal principal directions $e_1,\dots,e_{d-1}\in T\partial X$, can be defined a.e.\  as the eigenvalues of the differential $du$. 
Hence the second fundamental form $\II_x$ on the tangent space $T_x\partial X$  is defined for $\Ha^{d-1}$-a.a.\ $x\in \partial X$. For $\sum_{i=1}^{d-1}\alpha_ie_i\in T_x\partial X$, $\II_x$ is the quadratic form given by
\begin{equation*}
\II_x\left(\sum_{i=1}^{d-1}\alpha_ie_i \right)= \sum_{i=1}^{d-1}k_i(x)\alpha_i^2
\end{equation*}
whenever $d_xu$ is defined. In particular, the trace is $\tr  (\II) =k_1+\dotsm+k_{d-1}$. Note for later that $r$-regularity ensures that $k_1,\dots,k_{d-1}\leq r^{-1}$. 

The $(d-2)$nd curvature measure of $X$ is defined \cite{federer} for $r$-regular sets by 
\begin{equation*}
C_{d-2}(X;A)=\frac{1}{2\pi} \int_{\partial X\cap A } \tr(\II) d\Ha^{d-1}
\end{equation*}
for all Borel sets $A\subseteq \R^d$.

\section{Asymptotic formulas}\label{1sec}
\subsection{First order formulas}

The following notation will be used in the proofs.
For a finite set $S$ and an interval $I$, we denote by $I^S$ the $|S|$-tuples $\{\theta_s\}_{s\in S}$ of points $\theta_s \in I$ indexed by $S$.
Given a finite set $S\subseteq \R^d$ we write
\begin{align*}
\Theta_a^X(x;S)&=\{\theta_a^X(x+s)\}_{s\in S} \in [0,1]^S\\
\Theta_u(t;S)&=\{\theta_u(t+\langle s, u \rangle)\}_{s\in S}  \in [0,1]^S.
 \end{align*}

For $x\in \partial X$ understood and $u$ an outward pointing normal, we also write $H_u:=H_{\langle x,u\rangle,u}^-$ for the supporting halfspace.
Note that
\begin{align*}
\theta_a^{H_u}(x + a(tu +s)){}&=\theta^{H_{0,u}^-}_a(a(t+\langle s, u\rangle)u) = \theta_u(t+\langle s,u \rangle)\\
\Theta_u(t;S){}&= \{\theta_a^{H_u}(x+a(tu+s))\}_{s\in S}.
\end{align*}

The proofs follow from the following lemma shown in \cite[Lemma 7.1 and 7.2]{am4}:
\begin{lem}\label{thetadist}
Suppose $X$ is gentle and $\rho$ is a bounded PSF. Let $D>0$. Then for a.a.\ $x\in \partial X$,
\begin{equation*}
\sup\{|\theta^X_a(x+atu+as)-\theta_u(t+\langle s, u \rangle)| \mid t\in [-D,D], s\in B(D)\}\in o(1).
\end{equation*}
\end{lem}

\begin{thm}\label{1stgrey}
Suppose $X\subseteq \R^d$ is a compact gentle set, $S\subseteq \R^d$ is finite, and $\rho$ is a bounded PSF. Let $f: (0,1)^S \times \R^d\to \R$ be continuous with $\supp f \subseteq [\beta,\omega]^{S} \times \R^d$ for some $\beta, \omega \in (0,1)$. Then
\begin{align*}
\MoveEqLeft \lim_{a\to 0} a^{-1}\int_{\R^d} f( \Theta^X_a(x;aS),x)dx =\int_{\partial X} \int_{\R}  f( \Theta_u(t;S),x) dt\Ha^{d-1}(dx).
\end{align*}
\end{thm}

\begin{proof}
Let $D>0$ be such that 
\begin{equation*}
\int_{|x| \geq \frac{D}{2}} \rho(x) dx \leq \beta , 1-\omega
\end{equation*}
and $S \subseteq B\big(\frac{D}{2}\big)$ where $B(R)$ denotes the ball in $\R^d$ of radius $R$.
This ensures that $\supp f(\Theta^X_a(x;aS),x)  \subseteq \partial X \oplus B(aD)$. Then the generalized Weyl tube formula \cite[Theorem 2.1]{last} yields
\begin{align}\label{lastform}
 \int_{\R^d}  f(\Theta^X_a(x;aS),x) dx = {}& \sum_{m=1}^d m\kappa_m \int_{N(\partial X)}\int_0^{\delta(\partial X;x,u)} t^{m-1}\\
& \times \nonumber
  f(\Theta^X_a(x+tu;aS),x+tu) dt \mu_{d-m}(\partial X;d(x,u)).
\end{align}
Here $\kappa_m$ is the volume of the unit ball in $\R^m$ and the $\mu_i$ are certain signed measures of locally finite total variation. 

Observe that
\begin{equation}\label{vurdering}
\int_0^{\delta(\partial X;x,u)}  t^{m-1}f(\Theta^X_a(x+tu;aS),x+tu) dt  \leq m^{-1}a^mD^m\sup |f|
\end{equation}
so that dominated convergence together with \cite[Equation (8)]{rataj}  yields
\begin{align*}
\MoveEqLeft \lim_{a\to 0}a^{-1}\sum_{m=1}^d m\kappa_m \int_{N(\partial X)} \int_0^{aD} t^{m-1} f(\Theta^X_a(x+tu;aS),x+tu) dt \mu_{d-m}(\partial X;d(x,u)) \\
& = \int_{\partial X }\bigg( \lim_{a\to 0}\int_{-D}^{D}  f(\Theta^X_a(x+atu;aS),x+atu)  dt\bigg) \Ha^{d-1}(dx)\\
& = \int_{\partial X }\int_{-D}^{D}  f(\Theta_u(t;S),x)  dt\Ha^{d-1}(dx).
\end{align*}
The last equation follows from Lemma \ref{thetadist} and continuity of $f$.
\end{proof}

Assume  $X$ is compact gentle and $\rho $ bounded. Let $A\subseteq (0,1)^S$ be a compact set and $g:\R^d \to \R$ a continuous function. Define the measures on $A$ given for any Borel set $B\subseteq A$ by
\begin{equation*}
\mu^{X,g}_a(B)=a^{-1}\int_{\R^d} \mathds{1}_B\big( \Theta^{X}_a(x+atu;aS)\big) g(x) dx
\end{equation*}
and 
\begin{equation*}
\mu^{X,g}(B)=\int_{\partial X} \int_{-D}^D\mathds{1}_B\big( \Theta_u(t;S)\big)dtg(x)  \Ha^{d-1}(dx).
\end{equation*}

\begin{cor}
Let $X$ be a compact gentle set and $A\subseteq (0,1)^S$ a compact set.  Let $g: \R^d \to \R$ be continuous and assume $\mu^{X,g}(\partial A)=0$. Then $\mu_{a}^{X,g}$ converges weakly to $\mu^{X,g}$. In particular, if $h:A\to \R$ is continuous and $f(\Theta,x)=h(\Theta)g(x)$, then
\begin{equation*}
\lim_{a\to 0} E\hat{\Phi}_q^f (X) = \int_{\partial X} \int_{-D}^D h(\Theta_u(t;S)) dt g(x)  \Ha^{d-1}(dx).
\end{equation*}
\end{cor}

\begin{proof}
For  any bounded continuous function $h: A \to \R$,  
\begin{equation*}
\int_A h d\mu^{X,g}_a \to \int_A h d\mu^{X,g}.
\end{equation*}
This follows from Theorem \ref{1stgrey} by approximating $h$ by continuous functions on $(0,1)^S$.
\end{proof}


\subsection{Notation}
We next introduce some more notation that will be used in the statement of the main second order theorem and its proof in order to keep formulas short. Moreover, we state a technical lemma proved in \cite{am4}.

We will assume $\rho $ to be continuous and compactly supported. In this case all $\theta_u$ are $C^1$ with $(u,t)\mapsto \theta_u'(t)$ continuous.
We say that $\beta \in (0,1)$ is a regular value if $\theta_u'(t)<0$ for all $t$ with $\theta_u(t)=\beta$ and all $u\in S^{d-1}$. Since $\theta_u$ is decreasing, this ensures that $\theta_u^{-1}(\beta)$ is uniquely determined.

For $X\subseteq \R^d$ $r$-regular, define the quadratic approximation $Q_x$ to $X$ at $x\in \partial X$ by
\begin{equation*}
Q_x=\{z\in \R^d \mid \langle z-x,u\rangle  \leq -\tfrac{1}{2}\II_x(\pi_{u^\perp}(z-x)) \}
\end{equation*}
where $\pi_{u^\perp} :\R^d \to u^\perp$ denotes the orthogonal projections.

It is shown in \cite{am4}, in the proof of Lemma 7.6, that for $s\in \R^d$
\begin{equation}\label{Qexp}
\theta^{Q_x}_a(x+a(tu+s)) = \theta_u (t+\langle s,u\rangle ) +  a\theta^{Q_x} (t,s) +o(a)
\end{equation}
where
\begin{equation*}
\theta^{Q_x} (t,s)=-\frac{1}{2}\int_{u^\perp} \II_x(z)\rho(z-tu-s)dz.
\end{equation*}
Again we use the notation 
\begin{align*}
\Theta^{Q_x}_a (t;S){}&=\{\theta^{Q_x}_a (x+atu + s)\}_{s\in S} \\
\Theta^{Q_x} (t;S){}&=\{\theta^{Q_x}(t, s)\}_{s\in S} .
\end{align*}

Choose $D$ as in the proof of Theorem \ref{1stgrey}.
Given $A\subseteq (0,1)^S$ and $x\in \partial X$, let
\begin{align*}
t_0^S ={}& \inf\{t\in [-D,D]\mid \Theta_u(t;S)\in A\}\\
t_1^S ={}& \sup\{t\in [-D,D]\mid \Theta_u(t;S)\in A\}\\
t_0^S(a) ={}& \inf\{t\in [-D,D]\mid \Theta^{Q_x}_a(t;aS)\in A\}\\
t_1^S(a) ={}& \sup\{t\in [-D,D]\mid \Theta^{Q_x}_a(t;aS)\in A\}\\
t_0^{X,S}(a) ={}& \inf\{t\in [-D,D]\mid \Theta^{X}_a(x+atu;aS)\in A\}\\
t_1^{X,S}(a) ={}& \sup\{t\in [-D,D]\mid \Theta^{X}_a(x+atu;aS)\in A\}.
\end{align*}
Finally, let
\begin{align*}
\psi_0^S(x) ={} & \max \bigg\{ -\frac{\theta^{Q_x} (t_0^S,s) }{\theta_u' (t_0^S+\langle s,u\rangle)} \mid s \in S, t_0^S=t_0^s\bigg\}\\
\psi_1^S(x) ={} & \min \bigg\{ -\frac{\theta^{Q_x} (t_1^S,s)}{ \theta_u '(t_1^S+\langle s,u\rangle )} \mid s \in S, t_1^S=t_1^s\bigg\}.
\end{align*}

\begin{lem}\label{t4}
Suppose that $X$ is $r$-regular and $\rho$ is continuous with compact support. Let $R>0$ and $S\subseteq \R^d$ finite be given.

For all $a$ sufficiently small, $t\mapsto \theta^X_a(x+a(tu+s))$ and $t\mapsto \theta^{Q_x}_a(x+a(tu+s))$ are decreasing functions for all $x\in \partial X$, $s\in S$, and $t\in [-R,R]$.

There is a constant $M>0$ such that for $\nu =0,1$ and $a$ sufficiently small 
\begin{align}\label{XH1}
&\sup\Big\{\Big|\Theta^X_a(x+atu;aS)-\Theta_u(t;S)\Big| \mid  x\in \partial X, t\in [-R,R]\Big\} \leq  Ma\\
&\sup\Big\{ \Big|t^{X,S}_\nu(a)-t_\nu^S\Big| \mid x\in \partial X\Big\} \leq  Ma.\label{XH2}
\end{align}

Assume that $A=\bigtimes_{s\in S} I_s$ where $I_s$ are intervals and all points in $\partial I_s $ are regular values. Then for each $x\in \partial X$,
\begin{align}
\begin{split}\label{XQ}
&\sup\Big\{\Big|\Theta^X_a(x+atu;aS)-\Theta^{Q_x}_a(t;aS)\Big|\mid t\in [-R,R] \Big\} \in o(a)\\
&\Big|t_\nu^{X,S}(a)-t_\nu^S(a)\Big| \in o(a)
\end{split}
\end{align}
for $\nu =0,1$ and
\begin{equation}\label{texp}
t_\nu^S(a)= t_\nu^S + a \psi_\nu^S(x) +o(a).
\end{equation}
\end{lem}

\begin{proof}
The lemma is essentially proved in \cite{am4}. Note that the notation is changed.
The first statement is proved in Lemma 7.5 for $\theta^X$. The proof for $\theta^{Q_x}$ is similar. 
Equations \eqref{XH1} and \eqref{XH2} are shown in the proof of Theorem 3.2 and 5.2. Equation \eqref{XQ} follows from Lemma 7.7 and \eqref{texp} from Lemma 7.6.
\end{proof}

\subsection{Second order formulas}

\begin{thm}\label{2ndgrey}
Suppose $X$ is an $r$-regular set and $\rho$ is continuous and compactly supported. Let $S\subseteq \R^d$ be a finite set and $A=\bigtimes_{s\in S} I_s$ where $ I_s \subseteq (0,1)$ are closed intervals such that $\partial I_s$ consists of regular values for all $s\in S$. Let $f: A\times \R^d\to \R$ be $C^1$. Then
\begin{align*}
\MoveEqLeft \lim_{a\to 0} \bigg(a^{-2} \int_{\R^d}f( \Theta^X_a(x;aS),x)dx- a^{-1}\lim_{a\to 0} a^{-1} \int_{\R^d}f( \Theta^X_a(x;aS),x)dx\bigg)\\
 ={}&\int_{\partial X} \int_{t_0^S}^{t_1^S} tf( \Theta_u(t;S),x) dt\tr(\II_x)\Ha^{d-1}(dx)\\
 +{}&\int_{\partial X} \int_{t_0^S}^{t_1^S}\Big( \Big\langle \nabla^1 f(\Theta_u (t;S),x),\Theta_0^{Q_x} (t;S)\Big \rangle + t\Big\langle \nabla^2 f(\Theta_u (t;S),x),u\Big\rangle\Big) dt \Ha^{d-1}(dx)\\
 +{}& \int_{\partial X}\Big(f(\Theta_u (t_1^S;S),x)\psi_1^S(x)-f(\Theta_u (t_0^S; S),x)\psi_0^S(x)\Big)\Ha^{d-1}(dx).
\end{align*}
\end{thm}

Here $\nabla^1, \nabla^2$ are the gradients of $\Theta \mapsto f(\Theta, x)$ and $x \mapsto f(\Theta ,x)$, respectively.

\begin{proof}
For $r$-regular sets, the generalized Weyl tube formula reduces to 
\begin{align*}
\MoveEqLeft \int_{\R^d} f(\Theta_a^X(x;aS),x) dx\\
& = a\sum_{m=1}^d  \int_{\partial X } \int_{-D}^{D} t^{m-1} f (\Theta_a^X (x+atu;aS), x +atu )dt s_{m-1}(x) \Ha^{d-1}(dx).
\end{align*}
where $s_m(x)$ is the $m$th symmetric polynomial in the principal curvatures at $x$ whenever these are defined.

Again, \eqref{vurdering} shows that Lebesgue dominated convergence applies to all terms with $m \geq 2$ and shows that all terms with $m\geq 3$ vanish asymptotically.

For $m=2$ consider
\begin{align}\label{ulig1}
\MoveEqLeft \int_{-D}^{D} \Big|t f (\Theta_a^X (x+atu;aS),x+atu)- t f (\Theta_u (t;S),x)\Big|dt \\ \nonumber
 &\leq  2D^2\sup|\nabla f|\sup\big\{\big|\Theta_a^X (x+atu;aS)-\Theta_u (t;S)\big| + aD \mid \\
&\quad t\in \big[t_0^S,t_1^S \big]  \cap \big[t_0^{X,S}(a), t_1^{X,S}(a)\big]\big\}
  + 2D\sup |f|\Ha^1([t_0^S,t_1^S ]  \Delta [t_0^{X,S}(a), t_1^{X,S}(a)])\nonumber
\end{align}
where $\Delta$ denotes the symmetric difference.
By Equations \eqref{XH1} and \eqref{XH2}, the right hand side is of order $O(a)$.

For the  $m=1$ term, a similar argument shows that
\begin{equation}\label{mustconverge}
a^{-1}\int_{-D}^{D} \Big( f (\Theta_a^X (x+atu;aS),x+atu) -  f (\Theta_u (t;S),x)\Big)dt
\end{equation} 
is uniformly bounded. Hence another application of dominated convergence shows that it is enough to determine the limit of this for each $x\in \partial X$.

Another argument similar to \eqref{ulig1} using Equations \eqref{XQ} shows that
\begin{equation*}
\lim_{a\to 0} \int_{-D}^D a^{-1} \Big|f (\Theta_a^X (x+atu;aS),x+atu)-  f (\Theta_a^{Q_x} (t;aS),x+atu)\Big| dt =0.
\end{equation*}
Thus it remains to compute
\begin{equation*}
\lim_{a\to 0}\int_{-D}^D a^{-1}\Big( f (\Theta_a^{Q_x} (t;aS),x+atu)- f (\Theta_u (t;S),x)\Big)dt.
\end{equation*}

The integrand is uniformly bounded on
\begin{equation*}
G(a)= \big(t_0^S,t_1^S\big) \cap \big(t_0^S(a),t_1^S(a)\big)
\end{equation*}
by differentiability of $f$ and another application of Lemma \ref{t4} \eqref{XH1} with $X$ replaced by $Q_x$.
Observe that 
\begin{equation*}
\mathds{1}_{G(a)}(t)\to \mathds{1}_{\big(t_0^S,t_1^S\big)}(t)
\end{equation*}
 pointwise.
 Hence by dominated convergence and Equation \eqref{Qexp},
\begin{align*}
\MoveEqLeft \lim_{a\to 0} a^{-1}\int_{G(a)}\Big( f (\Theta_a^{Q_x} (t;aS),x+atu)-  f (\Theta_u(t;S),x)\Big)dt\\
&=\int_{-t_0^S}^{t_1^S} \Big(\Big\langle \nabla^1 f(\Theta_u (t;S),x),\Theta^{Q_x} (t; S)\Big\rangle +t\Big\langle \nabla^2 f(\Theta_u (t;S),x),u \Big\rangle\Big)dt.
\end{align*}

It remains to consider the integral over the sets
\begin{equation}\label{intervals}
\big[t_0^S(a)\wedge t_0^S,t_0^S(a)\vee t_0^S\big] \text{ and } \big[t_1^S(a)\wedge t_1^S,t_1^S(a)\vee t_1^S\big].
\end{equation}
The integral over the first set is
\begin{align*}
\MoveEqLeft -\int_{t_0^S}^{t_0^S(a)}a^{-1} \Big( f (\Theta_a^{Q_x} (t;aS),x+atu) +  f (\Theta_u (t;S),x)\Big) dt\\
&=
-\int_{t_0^S}^{t_0^S+a\psi_0^S(x)} a^{-1} \Big( f (\Theta_a^{Q_x} (t;aS),x+atu) +  f (\Theta_u (t;S),x)\Big) dt+o(1)
\end{align*}
by Lemma \ref{t4} \eqref{texp}.
Since $|t-t_0^S|\leq Ma$ for all $t \in [t_0^S(a)\wedge t_0^S,t_0^S(a)\vee t_0^S]$, 
\begin{align*}
\MoveEqLeft -\int_{t_0^S}^{t_0^S+a\psi_0^S(x)}a^{-1} \Big( f (\Theta_a^{Q_x} (t;aS),x+atu) +  f (\Theta_u (t;S),x)\Big) dt \\
&= -\int_{t_0^S}^{t_0^S+a\psi_0^S(x)} a^{-1} f (\Theta_u (t_0^S;S),x) dt +o(1)\\
&= -\psi_0^S(x) f (\Theta_u (t_0^S;S),x) dt +o(1).
\end{align*}
The second interval in \eqref{intervals} is treated similarly.
\end{proof}

\section{Estimation of Minkowski tensors}\label{mink}

\subsection{Minkowski tensors}\label{tens}
To a compact set $X\subseteq \R^d$, we associate 
 the generalized curvature measures  $C_k(X;\cdot) $ on $\Sigma=\R^d\times S^{d-1}$ for $k=0,\dots,d-1$, see \cite{schneider} in the case of poly-convex sets and \cite{federer} for sets of positive reach. An extension to general compact sets can be found in \cite{last}.  

Let $\mathbb{T}^p$ denote the space of symmetric tensors on $\R^d$ of rank $p$. Identifying $\R^d$ with its dual using the Euclidean inner product $\langle \cdot,\cdot \rangle$, one can interpret a symmetric $p$-tensor as a symmetric $p$-linear functional on $\R^{d}$. 
Let $x^r$ denote the $r$-fold symmetric tensor product of $x \in \R^d$.
For $X\subseteq \R^d$ and $k=0,\dots,d-1$, $r,s\geq 0$ we associate the $(r+s)$-tensors
\begin{equation*}
\Phi^{r,s}_{k}(X)= \frac{1}{r!s!}\frac{\omega_{d-k}}{\omega_{d-k+s}}\int_{\Sigma} x^r u^s C_k(X;d(x,u)),
\end{equation*}
and for $r\geq 0$ we define the volume tensors
\begin{equation*}
\Phi^{r,0}_{d}(X)= \frac{1}{r!} \int_{X} x^r \Ha^d(dx).
\end{equation*}
These are the so-called Minkowski tensors introduced in \cite{mcmullen}, see also e.g.\ \cite{hug,schuster}.

The Minkowski tensors satisfy the McMullen relations \cite{mcmullen} on convex sets, 
\begin{equation*}
2\pi\sum_s s \Phi_{k-r+s}^{r-s,s} = Q\sum_s \Phi_{k-r+s}^{r-s,s-2} 
\end{equation*}
where $k\geq 0$, $r\geq 0$, and $Q$ is the metric tensor. All tensors in the sum that have not yet been defined should be interpreted as 0.

Below we shall define estimators for $\Phi_d^{r,0}$, $\Phi^{r,s}_{d-1}$, and $\Phi_{d-2}^{r,0}$. In 2D, the McMullen relations show that all tensors are linear combinations of  multiples of these by powers of $Q$. Hence, in 2D we obtain a complete set of estimators for the Minkowski tensors.

\subsection{Volume tensors}\label{volume}
It is easy to see that the volume tensors can be estimated unbiasedly from black-and-white images even in finite resolution just using a Riemann sum:
\begin{equation*}
\hat{\Phi}^{r,0}_{d}(X)=a^d\frac{1}{r!}\sum_{z\in a\La\cap X} z^r.
\end{equation*}
If only a grey-scale image is given, one may threshold the image at level $\beta  \in (0,1)$ and apply this estimator. This yields the estimator
\begin{equation*}
\hat{\Phi}^{r,0}_{d}(X)=a^d\frac{1}{r!}\sum_{z\in a\La} \mathds{1}_{\{\theta_a^{X}(z)\geq \beta \} } z^r.
\end{equation*}
This is asymptotically unbiased for all sets with $\Ha^{d-1}(\partial X ) < \infty$ since
\begin{equation*}
E\hat{\Phi}^{r,0}_{d}(X)=\frac{1}{r!}\int_{\R^d}  z^r \mathds{1}_{\{\theta_a^{X}(z)\geq \beta \} } dz
\end{equation*}
and  $|\mathds{1}_{\{\theta_a^{X}\geq \beta \} }-\mathds{1}_X| \leq \mathds{1}_{\partial X \oplus B(aD)}$ where $D$ is such that $\int_{|z|\leq D} \rho(z)dz \geq \beta, 1-\beta $.
%
%
%
%

\subsection{Surface tensors}\label{surfacet}
In this section we define local algorithms based on $2\times \dotsm \times 2$ configurations for the surface tensors $\Phi^{r,s}_{d-1}(X)$. For gentle sets, these take the form
\begin{equation*}
\Phi^{r,s}_{d-1}(X)= \frac{1}{r!s!}\frac{2}{\omega_{s+1}}\int_{\partial X} x^r u^s \Ha^{d-1}(dx).
\end{equation*}
Identifying $\R^d$ with its dual, it is enough to determine all their evaluations on a basis $v_1,\dots, v_d$,
\begin{equation}\label{tensor}
\Phi^{r,s}_{d-1}(X)(v_{i_1},\dots, v_{i_{r+s}})= \frac{1}{r!s!}\frac{2}{\omega_{s+1}}\int_{\partial X} \prod_{k=1}^r \langle x,v_{i_k}\rangle \prod_{l=r+1}^{r+s} \langle u(x),v_{i_l}\rangle \Ha^{d-1}(dx)
\end{equation}
for all choices of ${i_1},\dots, {i_{r+s}} \in \{1,\dots ,d\}$.  Hence it is enough to estimate \eqref{tensor} for each tuple ${i_1},\dots, {i_{r+s}}$. 
As basis we choose the vectors $v_1,\dots, v_d$ spanning $\La$. Let $V=\max\{|v_i|, i=1\dots, d\}$.

As in the case of surface area estimators, this requires some assumptions on the $PSF$:
\begin{enumerate}[label=(\roman*)]
 \item \label{i} $\rho$ is rotation invariant $\rho(x)=\rho(|x|)$. In this case, $\theta_u(t):=\theta(t)$ is independent of $u$ and 
\begin{equation*}
\theta_u(t;S)= \{\theta(t+\langle u,s \rangle)\}_{s\in S}.
\end{equation*} 

\item \label{ii} $\theta$ is strictly decreasing on $\theta^{-1}(0,1)$. In this case the inverse exists on $(0,1)$ and we denote this by $\varphi$.

\item \label{iii} The lattice is so fine compared to the support of $\rho$ that $\theta^{-1}(0,1)$ contains an interval of the form $[\beta-V, \omega + V]$ where $\beta < \omega$. In particular, $\varphi$ is well-defined on $[\beta-V, \omega + V]$.
\end{enumerate}
Note that (i) and (ii) are satisfied for both the Gaussian and the Airy disk.

Under these conditions, observe that 
\begin{equation}\label{coordinates}
\varphi(\theta_u(t+\langle v_i, u \rangle ))-\varphi(\theta_u(t)) = \langle u, v_i \rangle.
\end{equation}
for $t\in [\beta,\omega]$. Let $S=\{0,v_1,\dots,v_d\}\subseteq C_{0,0}^2$ and $A=[\beta, \omega] \times \bigtimes_{s\in S \backslash \{0\}} [\beta-V, \omega + V]$.
Define the weight function
\begin{equation}\label{deff}
f\big(\{\theta_s\}_{s\in S},x \big) = \mathds{1}_{ A}\big(\{\theta_s\}_{s\in S}\big) \frac{1}{r!s!}\frac{2}{\omega_{s+1}}\prod_{k=1}^{r} \langle x, v_{i_k} \rangle \prod_{l=r+1}^{r+s} \big(\varphi(\theta_{v_{i_l}})-\varphi(\theta_0)\big)
\end{equation}
This requires of course that $\varphi$ is known, or equivalently, the blurring of a halfspace $\theta$.

Applying Theorem \ref{2ndgrey} to the local estimator with weight function \eqref{deff} yields:
\begin{cor}
Let $X$ be a gentle set and suppose $\rho$ satisfies Condition \ref{i}--\ref{iii}. If $f$ is as in \eqref{deff}, then 
\begin{equation*}
\lim_{a\to 0 } E\Phi^f_{d-1}(X) = 
(\varphi(\beta)- \varphi(\omega) )\frac{1}{r!s!}\frac{2}{\omega_{s+1}} \int_{\partial X}\prod_{k=1}^{r} \langle x, v_{i_k} \rangle \prod_{l=r+1}^{r+s} \langle u(x), v_{i_l}  \rangle \Ha^{d-1}(dx). 
\end{equation*}
Since $\varphi$ is strictly decreasing, $(\varphi(\beta)- \varphi(\omega) )> 0$. Dividing by this factor thus yields an asymptotically unbiased estimator for \eqref{tensor}. 
\end{cor}

For $r$-regular sets, a formula for the first order bias is given by Theorem \ref{2ndgrey}. Using $3\times \dotsm \times 3$ configurations instead, we can make the first order bias vanish. Let $S=\{0, \pm v_1,\dots,\pm v_d\}\subseteq C_{v,0}^3$ where $v=v_1+\dotsm + v_d$. Consider the weight function
\begin{equation}
\begin{split}\label{deff2}
f\big(\{\theta_s\}_{s\in S} ,x\big) = {}&\mathds{1}_{A}\big(\{\theta_s\}_{s\in S}\big) \frac{1}{r!s!}\frac{2}{\omega_{s+1}}\prod_{k=1}^{r} \langle x, v_{i_k} \rangle\\
&\times \bigg(\prod_{l=r+1}^{r+s} (\varphi(\theta_{v_{i_l}})-\varphi(\theta_0)) +  \prod_{l=r+1}^{r+s} (\varphi(\theta_0)-\varphi(\theta_{-v_{i_l}}))\bigg)
\end{split}
\end{equation}
where
\begin{equation*}
A=[\beta, 1-\beta ] \times \bigtimes_{s\in S\backslash \{0\}} [\beta-V-\eps, 1-\beta + V+\eps]
\end{equation*}
for $\eps>0$ so small that $[\beta-V-\eps, 1-\beta + V+\eps]\subseteq \theta^{-1}(0,1)$. Then Theorem \ref{2ndgrey} yields:

\begin{cor}
Let $X$ be an $r$-regular set and suppose $\rho$ satisfies Condition \ref{i}--\ref{iii}. If $f$ is as in \eqref{deff2}, then 
\begin{equation*}
E\Phi^f_{d-1}(X) = 
(\varphi(\beta)- \varphi(\omega) ) \frac{1}{r!s!}\frac{2}{\omega_{s+1}}\int_{\partial X}\prod_{k=1}^{r} \langle x, v_{i_k} \rangle \prod_{l=r+1}^{r+s} \langle u(x), v_{i_l}  \rangle \Ha^{d-1}(dx)+o(a). 
\end{equation*}
\end{cor}

\begin{rem}
More generally, $u$ is determined by its coordinates \eqref{coordinates} in the basis $v_1,\dots ,v_d$. This can be used in a similar way to find estimators for integrals of the form
\begin{equation*}
\int_{\partial X} f(x,u(x)) \Ha^{d-1}(dx).
\end{equation*}
 \end{rem}

\begin{rem}
Since $\tr \big(\Phi_{d-1}^{0,2}(X)\big) $ is just the surface area of $X$ up to a constant factor, the above also yields a new surface area estimator. Taking larger configurations into account than the surface area estimators in \cite{am4}, one could hope for a better precision. On the other hand, this new estimator requires more knowledge about the underlying PSF and is hence harder to apply in practice.
\end{rem}

\begin{rem}
It is known that asymptotically unbiased local surface area estimators from black-and-white images do not exist \cite{am3}. Tensors of the form $\Phi_{d-1}^{r,1}$ can be estimated, but in general, asymptotically unbiased local estimators for $\Phi^{r,s}_{d-1}$ are not expected to exist for $s>0$. 
\end{rem}

%
%
%
%

\subsection{Mean curvature tensors}
We similarly obtain estimators for tensors of the form $\Phi_{d-2}^{r,0}$. Let $\beta\in \big(0,\frac{1}{2}\big)$ and let $g:[\beta,1-\beta] \to \R$ be a $C^1$ function satisfying $g(x)=-g(1-x)$. Define
\begin{equation}\label{fg}
f(\theta_0,x)= g(\theta_0)x^r. 
\end{equation}
This defines a local estimator $\hat{\Phi}_{d-2}^{f}$.

Theorem \ref{1stgrey} and \ref{2ndgrey} yield:
\begin{cor}
Suppose $X$ is a compact $r$-regular set and $\rho$ is continuous with compact support and satisfies Condition \ref{i}--\ref{ii} in Section \ref{surfacet}. With $f$ as in \eqref{fg}
\begin{equation*}
\lim_{a\to 0} E\hat{\Phi}_{d-2}^{f}(X)=2\pi r! (c_1+c_2+c_3)\Phi_{d-2}^{r,0}(X) + r!\int_{-\varphi(\beta)}^{\varphi(\beta)} t g(\theta(t)) dt \Phi_d^{r,0}(X).
\end{equation*}
where the constants $c_1,c_2,c_3 \in  \R $ are as in \cite[Section 6.2]{am4}.
\end{cor}
This follows by rewriting the limit in Theorem \ref{2ndgrey} exactly as in \cite{am4}. The $\Phi_d^{r,0}(X)$-term comes from  the $\nabla^2$-term by an application of the divergence theorem.  We already found asymptotically unbiased estimators for volume tensors in Subsection~\ref{volume}, so this can be corrected for. 
Estimators for which $c_1+c_2+c_3 \neq 0 $ are suggested in \cite[Section 6.2]{am4}. 
%
%
For instance, this is the case for $g(\theta)= (\theta-\frac{1}{2})\mathds{1}_{[\beta,1-\beta]}(\theta)$ and $g(\theta)= \mathds{1}_{[\beta,\frac{1}{2}]}(\theta)-\mathds{1}_{[\frac{1}{2},1-\beta]}(\theta)$ and for suitable values of $\beta$.

The remaining mean curvature tensors seem to be harder to get a hold of, since the asymptotic mean involves the surface normals in a more involved way than in the case of surface tensors. 

\section{Acknowledgements}
The author was funded by a grant from the Carlsberg Foundation and hosted by the Institute of Stochastics at Karlsruhe Institute of Technology. The author would also like to thank Markus Kiderlen for helpful input and suggestions.

\end{document}